\documentclass{amsart}

\usepackage{amsthm, amsfonts, amssymb, amsmath, latexsym, enumerate, times}
\usepackage{amscd}
\usepackage{xypic}  
\usepackage{amssymb}
\usepackage{amsthm}
\usepackage{epsfig}
\usepackage{paralist}
\usepackage{lmodern}
\usepackage[T1]{fontenc}
\usepackage[utf8]{inputenc}
\usepackage{mathrsfs}
\usepackage{array}
\usepackage{comment}
\usepackage{paralist}
\usepackage{color}

\newtheorem{theorem}{Theorem}[section]
\newtheorem{lemma}[theorem]{Lemma}

\newtheorem{corollary}[theorem]{Corollary}
\newtheorem{proposition}[theorem]{Proposition}

\theoremstyle{definition}

\newtheorem{remark}[theorem]{Remark}

\renewcommand{\O}{{\mathcal O}}
\renewcommand{\L}{{\mathcal L}}

\renewcommand{\L}{{\mathcal L}}

\newcommand{\p}{{\mathbb P}}

\def\geq{\geqslant}
\def\leq{\leqslant}

\begin{document}
 
\title[On varieties whose surface sections have negative Kodaira dimension]{On varieties whose general surface section has negative Kodaira dimension}

\author{Ciro Ciliberto}
\address{Dipartimento di Matematica, Universit\`a di Roma Tor Vergata, Via O. Raimondo 00173 Roma, Italia}
\email{cilibert@mat.uniroma2.it}

\author{Claudio Fontanari}
\address{Dipartimento di Matematica, Universit\`a di Trento, Via Sommarive, 14, 38123 Povo, Italia}
\email{claudio.fontanari@unitn.it}

 
 
\maketitle

\medskip

\begin{abstract} In this paper, inspired by work of Fano, Morin and Campana--Flenner, we give a full projective classification of (however singular) varieties of dimension 3 whose general hyperplane sections have negative Kodaira dimension, and we partly extend such a classification to varieties of dimension $n\geq 4$ whose general surface sections have negative Kodaira dimension. In particular we prove that a variety of dimension $n\geq 3$ whose general surface sections have negative Kodaira dimension is birationally equivalent to the product of a general surface section times $\p^{n-2}$ unless (possibly) if the variety is a cubic hypersurface. \end{abstract}

\section{Introduction} 

In this paper we study irreducible, projective, non--degenerate, linearly normal, however singular, complex varieties $V\subseteq \p^r$ of dimension $n$, with $r\geq n+1$, such that the general surface  section of $V$ is rational or, more generally, has negative Kodaira dimension.

The subject has a long history that goes back more than a century ago. The first one who dealt with this topic was G. Fano in 1918, who claimed in \cite {Fa} that all threefolds with rational hyperplane sections are rational except the cubic threefold (at that time it was still unclear whether this was rational or not). The problem was taken over again about twenty years later by U. Morin in \cite {Mo}. In this paper Morin proposed a full projective classification of all varieties of dimension $n\geq 3$  whose general surface section is rational. Unfortunately both papers by Fano and Morin are affected by a serious mistake, which in modern terms can be explained as follows. They basically assume that, given a big and nef divisor $D$ on a smooth variety $X$, and given any positive integer $h$, one has $h^1(X, h(D+K_X))=0$. For $h=1$ this is ensured by Kawamata--Viehweg theorem, but for $h>1$, this is false. Despite this mistake, as we will see, the results claimed by Fano and Morin are true. 

The problem in question was reconsidered by Martynov in 1970 and 1973, first in the paper \cite {Ma1}, then in \cite {Ma2}. In the former paper Martynov proved that a smooth variety whose general hyperplane section is a rational surface is birationally equivalent either to $\p^3$ or to a cubic threefold. In the latter paper Martynov considers  the case of smooth threefolds whose general hyperplane section is a ruled surface with irregularity $q > 0$, proving that such a threefold is birationally equivalent to the  product of the projective plane and a nonsingular curve of genus $q$. The result in \cite {Ma2} has been reobtained, and partly extended to the singular case, by Buruiana \cite {Bu}. 

The last contribution on this topic is \cite {CF}, due to Campana--Flenner in 1993, where they classify pairs $(X,F)$, where $X$ is a normal threefold and $F$ is a smooth Cartier divisor on $X$, with Kodaira dimension $\kappa(F)=-\infty$ and ample normal bundle in $X$. Using Mori's theory, they prove that:\\
\begin{inparaenum}
\item [$\bullet$] either $X$ is birational to $F\times \p^1$;\\
\item [$\bullet$] or $X$ is birational to a sextic in $\p(1,1,1,2,3)$ or in $\p(1,1,2,2,3)$ with at most terminal singularities;\\
\item [$\bullet$] or $X$ is birational to a quartic  in $\p(1,1,1,1,2)$ with at most terminal singularities;\\
\item [$\bullet$] or $X$ is birational to a cubic threefold.  
\end{inparaenum}

The present paper can be considered as a set of footnotes to the aforementioned papers by Fano, Morin and Campana--Flenner. 

Our first objective has been to recover the classification result by Morin on threefolds whose general hyperplane section is rational. As we said before, Morin's argument contains a serious gap. However, by using ideas from \cite {CF} together with Morin's ones, we have been able to fix it. The result is the following:

\begin{theorem}\label{thm:main1} Let $V\subseteq \p^r$, with $r\geq 4$, be an irreducible, projective, non--degenerate, linearly normal variety of dimension 3, such that its general hyperplane section is a rational surface. Then $V$ is one of the following varieties (or internal projections of them):	\\
\begin{inparaenum}
\item [(i)] it is a threefold of minimal degree $d$ in $\p^{d+2}$, i.e., either a  quadric in $\p^4$, or a rational normal scroll or a cone over the Veronese surface of degree 4 in $\p^6$;\\
\item [(ii)] it is swept out by a 1--dimensional rational family of Veronese surfaces of degree 4 (or external projections of such surfaces);\\
\item [(iii)] it is swept out by a 1--dimensional rational family of (generically smooth) 2--dimensional quadrics;\\
\item [(iv)] it is a scroll in lines over a rational surface (in particular it could be a cone over a rational surface);\\
\item [(v)] it has degree $d$, with $3\leq d\leq 8$ and $r=d+1$, has genus 1 curve sections, and, if it is not a cone (in which case it falls in case (iv) above), it can have at most double points. If it is smooth, it is one of the following:\\
\begin{inparaenum}
\item [(v1)] a cubic hypersurface in $\p^4$;\\
\item [(v2)] the complete intersection of two quadrics in $\p^5$;\\
\item [(v3)] the section of the Grassmannian $\mathbb G(1,4)\subset \p^9$ of lines in $\p^4$ with a linear space of dimension 6;\\
\item [(v4)] the Pl\"ucker embedding of $\p^1\times \p^1\times \p^1$ in $\p^7$;\\
\item [(v5)] a hyperplane section of the Pl\"ucker embedding of $\p^2\times \p^2$ in $\p^8$;\\
\item [(v6)] the degree 7 image in $\p^8$ of the blow--up of $\p^3$ at a point $p$ via the morphism determined by the proper transform of the linear system of quadrics in $\p^3$ passing through $p$;\\
\item [(v7)] the 2--Veronese embedding of $\p^3$ in $\p^9$ that has degree 8;\\
\end{inparaenum} 
\item [(vi)] it is the 3--Veronese embedding of $\p^3$ in $\p^{19}$;\\
\item [(vii)] it is the 2--Veronese embedding of a quadric in $\p^4$ in $\p^{13}$;\\
\item [(viii)] it is the 2--Veronese embedding of the cone (with vertex a point) over a Veronese surface of degree 4 in $\p^5$, that has degree 32 and sits in $\p^{21}$;\\
\item [(ix)] it is the complete intersection of degree 8 in $\p^7$ of a cone with vertex a line $\ell$ over a Veronese surface of degree 4 in $\p^5$, with a quadric not containing the line $\ell$;\\
\item [(x)] it is the image of a cone in $\p^6$ with vertex a point over a rational normal scroll surface of degree 4 in $\p^5$ with a line directrix, via the linear system cut out on the cone by the cubic hypersurfaces containing three given plane generators of the cone: this variety has degree 36 and sits in $\p^{22}$;\\
\item [(xi)] it is the variety of degree 9 in $\p^7$ that is cut out on a cone with vertex a line over a rational normal scroll surface of degree 4 in $\p^5$ with a line directrix, by  a cubic hypersurface containing three 3--dimensional linear spaces generators of the cone.
\end{inparaenum}

In all the above cases $V$ is rational except if $V$ is a smooth cubic threefold in $\p^4$ in which case it is unirational but not rational. 
\end{theorem}

We notice that in item (i) we put the linearly normal threefold with rational curve linear sections (see \cite {EH}), that of course have rational surface linear sections. 

The proof of this theorem is contained in \S\S \ref {sec:applrat} and \ref {sec:fano}, after \S \ref {sec:basic} in which we give some basic definitions and \S \ref {sec:contr} that is devoted to recall some important preliminary results from \cite {CF}.

It is interesting to remark that Morin lists in \cite [\S 23]{Mo} all the types of linear systems of rational surfaces in $\p^3$ that represent (up to Cremona transformations) the rational threefolds in the list of Theorem \ref {thm:main1}. In this list there are three infinite series of linear systems, precisely:\\
\begin{inparaenum}
\item [($\alpha$)] linear systems of surfaces of degree $d$ with a point or a line of multiplicity $d-1$, that represent rational normal scrolls in item (i) of the list of Theorem \ref {thm:main1} or threefolds in item  (iv) of the list of Theorem \ref {thm:main1};\\
\item [($\beta$)] linear systems of surfaces of degree $d$ with a line of multiplicity $d-2$, that represent threefolds in item (ii) of the list of Theorem \ref {thm:main1};\\
\item [($\gamma$)] linear systems of surfaces of degree $d$ with a line $\ell$ of multiplicity $d-2$,
and a further base curve that is cut out in two points by the planes through $\ell$, 
 that represent threefolds in item (iii) of the list of Theorem \ref {thm:main1}.
 \end{inparaenum}
 
 Theorem \ref {thm:main1} implies that, besides these three types of linear systems (and linear systems contained in them or Cremona equivalent systems), there are only finitely many types of linear systems of rational surfaces in $\p^3$ (up to Cremona equivalence), and their dimensions do not exceed 22 (see item (x) in the list of Theorem \ref {thm:main1}). According to Morin \cite [\S 20]{Mo}, a Cremona representative of a linear system of dimension 22 of rational surfaces is as follows: \\
 \begin{inparaenum}
\item [($\delta$)]  linear system of surfaces of degree 6 in $\p^3$ with a point $p$ of multiplicity 4, with tangent cone consisting of a plane $\pi$ with multiplicity 4, with two infinitely near lines in the two successive infinitesimal neighborhoods of $p$ along $\pi$, and a further double point $q\not\in \pi$ with an infinitely double point, and a double line in the neighborhood of the infinitely near double point. Explicitly, such a linear system can be written in affine coordinates $(x,y,z)$ as
$$
az^2+z(yf_1+f_3)+y^3+y^2f_2+yf_4+f_6=0
$$  
where $f_i=f_i(x,y)$ is a homogeneous form of degree $i$ in $(x,y)$, for $i=1,2,3,4,6$. Here $p$ is the point at infinity of the $z$ axes, $\pi$ is the plane at infinity, $q$ is the origin and the infinitely near point to $q$ is along the $x$ axis. 
 \end{inparaenum}
 
 So we can state the:
 
 \begin{corollary}\label{cor:main} Besides the linear systems of indefinitely increasing dimensions of  types $(\alpha)$, $(\beta)$ and $(\gamma)$ (and systems contained in them or Cremona equivalent to them), the remaining linear systems of rational surfaces in $\p^3$ that determine maps birational to the image have dimension bounded by 22 and the bound is attained only by linear systems of the type $(\delta)$  (and Cremona equivalent systems). 
  \end{corollary} 
  
  This Corollary gives a negative answer to a conjecture by Castelnuovo (see \cite [pp. 187--188]{Ca}, or \cite  [pp. 415--416]{Ca1}), to the effect that a linear system of rational surfaces of dimension $r>19$ is Cremona equivalent to a system contained in one of type $(\alpha)$, $(\beta)$ or $(\gamma)$. Castelnuovo also suggests that the only linear systems of rational surfaces of dimension $19$ not contained in one  of type $(\alpha)$, $(\beta)$ or $(\gamma)$ are Cremona equivalent to the linear system of cubic surfaces in $\p^3$. But this, as we saw, it is not true.  

Morin gives in \cite [\S 37]{Mo} a full classification of varieties of dimension $n\geq 4$ whose general surface section is rational, and he proves that these varieties are all rational with the only possible exceptions of cubic hypersurfaces. Unfortunately Morin's arguments are not convincing and  we have not been able to completely fix them. Therefore we have  only the following result:

\begin{theorem}\label{thm:main2}  Let $V\subseteq \p^r$, with $r\geq 5$, be an irreducible, projective, non--degenerate, linearly normal variety of dimension $n\geq 4$, such that its general surface linear section is rational. Then $V$ is rational except, may be, if $V$ is a smooth cubic hypersurface in $\p^{n+1}$. 
\end{theorem}

The proof of this theorem, obtained in \S \ref {sec:thm2},  consists in a case by case analysis according to the various types of threefolds linear sections of the variety $V$ as listed in Theorem \ref {thm:main1}. It therefore follows from a series of different propositions (Propositions \ref {prop:ell}, \ref {prop:1dim}, \ref {prop:case3}, \ref {prop:fuji}, \ref {prop:fuji0}, \ref {prop:fuji1}, \ref {prop:veroint}, \ref {prop:lpic}). Although we have been unable to find a full projective classification of varieties $V\subseteq \p^r$ of dimension $n\geq 4$ such that their general surface section is rational, in some cases we succeeded in providing, in the aforementioned propositions,  some  information about such a classification. 

Then, we turn to the case of threefolds whose general hyperplane section has negative Kodaira dimension but is not rational. The following classification theorem follows right away from the Contraction Theorem \ref {thm:mori}. 

\begin{theorem}\label{thm:main3} Let $V\subseteq \p^r$, with $r\geq 4$, be an irreducible, projective, non--degenerate, linearly normal variety of dimension $3$, such that its general hyperplane section  $F$ is not rational with negative Kodaira dimension. Then $V$ is of one of the following types (or internal projections of such varieties):	\\
\begin{inparaenum}
\item [(a)] it is a scroll in planes parametrized by a curve of positive genus;\\
\item [(b)] it is swept out by a 1--dimensional family of Veronese surfaces of degree 4 (or external projections of such surfaces) parametrized by a curve of positive genus;\\
\item [(c)] it is swept out by a 1--dimensional  family of (generically smooth) 2--dimensional quadrics parametrized by a curve of positive genus;\\
\item [(d)] it is a scroll in lines over a surface birational to $F$.
\end{inparaenum}

 In any case $V$ is birational to $F\times \p^1$. 
\end{theorem}

Finally we consider varieties of dimension $n\geq 4$ whose general surface section has negative Kodaira dimension but is not rational. Again we have been unable to find a full projective classification of these varieties, but we could prove the following result (see \S \ref {sec:04}) that explains the birational structure of these varieties and gives some partial results about the aforementioned classification:

\begin{theorem}\label{thm:main4} Let $V\subseteq \p^r$, with $r\geq n+1$, be an irreducible, projective, non--degenerate, linearly normal variety of dimension $n\geq 4$, such that its general surface linear section $F$ is not rational with negative Kodaira dimension. Then $V$ is birational to $F\times \p^{n-2}$. 

In particular, if the general threefold section of $V$ falls in case (a) [resp. in case (d)] of the list in Theorem \ref {thm:main3}, then $V$ is a scroll in linear spaces of dimension $n-1$ parametrized by a curve of positive genus [resp. is a scroll in linear spaces of dimension $n-2$ parametrized by a non--rational surface with negative Kodaira dimension]. 
\end{theorem}

{\bf Acknowledgements:} The authors are members of GNSAGA of the Istituto Nazionale di Alta Matematica ``F. Severi''. This research project was partially supported by PRIN 2017 ``Moduli Theory and Birational Classification''.

\section{Basic definitions} \label{sec:basic}

In this paper we will be  interested in irreducible, projective, non--degenerate, linearly normal varieties $V\subseteq \p^r$ of dimension $n\geq 3$, such that their general 2--dimensional linear sections (with general linear subspaces of $\p^r$ of dimension $r-n+2$) have negative Kodaira dimension, in particular are rational. In a first part of this paper we will focus on the case $n=3$ and we will take the following equivalent viewpoint. We will consider pairs $(X, \L)$, where:\\
\begin{inparaenum}
\item [$\bullet$] $X$ is an irreducible, projective variety of dimension 3, having at most $\mathbb Q$--factorial, terminal singularities (in particular $X$ may have only isolated singularities);\\
\item [$\bullet$] $\L$ is a complete, base points free linear system of dimension $r$ of Cartier divisors on $X$, such that the morphism $\varphi_\L: X\longrightarrow \p^r$ determined by $\L$ is birational to its image $V$; then the general surface $F\in \L$ is smooth (and contained in the smooth locus of $X$), irreducible, and it is a big and nef divisor on $X$;\\
\item [$\bullet$] there is a non--empty open subset $U$ of $\L$ such that all surfaces $F\in U$ are smooth with negative Kodaira dimension.  
\end{inparaenum}

Pairs as $(X, \L)$ as above will be called \emph{R--pairs}. If for $F\in \L$ general, $F$ is rational, we will say that the R--pair $(X, \L)$  is \emph{rational}, otherwise it will be said to be \emph{irrational}.

Two R--pairs $(X,\L)$ and $(X',\L')$ are said to be \emph{birationally equivalent}, if there is a birational map $f: X\dasharrow X'$ that maps the linear system $\L$ to the linear system $\L'$. 

Given a R--pair pair $(X, \L)$, we will say that it is \emph{non--minimal} if there is pair $(X',\L')$ with a morphism $f: X\longrightarrow X'$ such that one of the following  facts happen:\\
\begin{inparaenum}
\item [$\bullet$] $f: X\longrightarrow X'$ is the blow--up of a subvariety $Z$ of $X'$  and $\L$ is the strict transform of $\L'$ via $f$;\\
\item [$\bullet$] $f: X\longrightarrow X'$ is the blow--up of a smooth point $p\in X'$ such that the general surface $F'\in \L'$ containing $p$ is smooth with negative Kodaira dimension and $\L$ is the strict transform of the linear system $\L'(-p)$ of surfaces in $\L'$ containing $p$.
\end{inparaenum}

A pair  $(X, \L)$ is said to be \emph{minimal} is it is not non--minimal. 
Given any R--pair, we can find a minimal R--pair birationally equivalent to it. So we can limit ourselves to consider minimal R--pairs. 
 
 \section{The contraction theorem}\label{sec:contr}

The following theorem will play a central role in what follows. 

\begin{theorem}[The Contraction Theorem]  \label{thm:mori} Let $(X_0,\L_0)$ be a minimal R--pair. Then there is a birationally equivalent R--pair $(X,\L)$ (with $F\in \L$ general)  with an extremal contraction $\phi: X\longrightarrow Y$ of fibering type with general fibre $Z$ such that either\\
\begin{inparaenum}
\item [(1)] $\phi: X\longrightarrow Y$ is a $\p^2$--bundle over a smooth curve $Y$, so that $Z\cong \p^2$,  and  $F^2\cdot Z\in \{1,2\}$, or\\
\item [(2)] $X$ is a quadric in a $\p^3$--bundle over a smooth curve $Y$, all fibres are irreducible and reduced, and $Z\cong \p^1\times \p^1$ with $\L_{|Z}\cong \O_{\p^1\times \p^1}(1,1)$, or\\
\item [(3)] $\phi: X\longrightarrow Y$ is a $\p^1$--bundle over a smooth surface, so that $Z\cong \p^1$, and $Z\cdot F=1$, or\\
\item[(4)] $X=Z$ is a $\mathbb Q$--Fano variety (i.e., $-K_X$ is ample) with Picard number $\rho(X)=1$.
\end{inparaenum}

\end{theorem}

This is basically \cite [Thm. (1.3)] {CF}. For the reader's convenience we will sketch the proof later in this section. 
Before proving it, we need a few preliminary results. 

\subsection{Preliminary results}  We recall some results from \cite {CF}. First of all, the following lemma is like \cite [Thm. (1.7)] {CF} and can be proved in the same way.

\begin{lemma}\label{lem:uno} Let $(X,\L)$ be an R--pair with $F\in \L$ general. Then there is a $t\in \mathbb Q$, with $0\leq t<1$, and there is an extremal ray $\mathbb R^+[C]$ on $X$ such that $C\cdot (F+tK_X)=0$ and $C\cdot (F+K_X)<0$. Accordingly one can consider the extremal contraction $\phi: X\longrightarrow Y$ of the ray $\mathbb R^+[C]$.
\end{lemma}

The following lemma is like  \cite [Cor. (2.3)] {CF}. We sketch the proof because there is some small difference between our situation here and the one in  \cite {CF}.

\begin{lemma}\label{lem:tre} Let $(X,\L)$ be a minimal R--pair with $F\in \L$ general. Consider the 
extremal contraction $\phi: X\longrightarrow Y$ of the ray $\mathbb R^+[C]$ as in Lemma \ref {lem:uno}. Assume that $\phi$ is birational with exceptional set $E$. If $\dim(E)=1$ then $E\cap F=\emptyset$. If $E$ is a (prime) divisor intersecting $F$, then $\phi(E)$ is a point.
\end{lemma}

\begin{proof} If $\dim(E)=1$, then one proves, as in \cite [Cor. (2.3)] {CF}, that $E\cdot F=0$. On the other hand $\L=|F|$ is base point free, so $E\cap F=\emptyset$. 

If $\dim(E)=2$ and $E\cap F\neq \emptyset$, suppose by contradiction that $\phi(E)$ is a curve and let $C$ be a generic fibre of $\phi:E\longrightarrow \phi(E)$. As before, one proves that $F\cdot C=0$. Then $F$ must intersect $E$ in fibres of $\phi:E\longrightarrow \phi(E)$. So the general $F$ contains a curve $C$. Then $(C^2)_F<0$. Moreover $K_X\cdot C<0$ (by Lemma \ref {lem:uno}). Then $K_F\cdot C=(F+K_X)\cdot C<0$. This implies that $C$ is a $(-1)$--curve on $F$ and this clearly contradicts the minimality of $(X,\L)$, because $\phi: X\longrightarrow Y$ turns out to be the blow--up of $Y$ along $\phi(E)$.
\end{proof}

The following lemma is inspired to \cite [Prop. (1.9)]{CF}. 

\begin{lemma}\label{lem:4} Let $(X,\L)$ be a minimal R--pair with $F\in \L$ general. Consider the 
extremal contraction $\phi: X\longrightarrow Y$ of the ray $\mathbb R^+[C]$ as in Lemma \ref {lem:uno}. Assume that $\phi$ is birational with exceptional set $E$. Then $E$ has empty intersection with $F$. 
\end{lemma}

\begin{proof} We assume, by contradiction, that $F$ has non--empty intersection with $E$. By Lemma \ref {lem:tre}, we may assume that $\dim(E)=2$ and that $\phi(E)$ is a point $p$. 

Let $\Gamma$ be the effective Cartier divisor cut out on $E$ by $F$. Since $\L$ is base point free, then $\Gamma$ is reduced by Bertini's theorem. More precisely, either $\Gamma$ is irreducible with $\Gamma \cdot _E\Gamma>0$ ($\cdot _E$ means that we make the intersection product of Cartier divisors on $E$, whereas $\cdot$ indicates the intersection product on $X$), 
or $\Gamma$ is the sum of irreducible reduced curves $\Gamma_1,\ldots, \Gamma_h$ belonging to the same base point free pencil on $E$, and in this case  $\Gamma_i\cdot _E\Gamma_j=0$ for all $1\leq i\leq j\leq h$, thus $\Gamma \cdot _E\Gamma=0$. 

Let $D$ be an irreducible component of $\Gamma$. Since $D$ is exceptional on $F$, one has $D\cdot _FD<0$. Moreover $D$, being a curve on $E$, it is contracted to a point by $\phi$, so it is numerically equivalent to $C$. Hence $D\cdot (F+K_X)<0$, i.e., $D\cdot_F K_F<0$. This implies that $D$ is a smooth rational $(-1)$--curve on $F$. 

Now we claim that it cannot be the case that $D\cdot_ED =0$. Indeed, if this happens, then one has $F\cdot D=0$. Since $D\cdot (F+K_X)<0$, we have $D\cdot K_X<0$. As $D\cdot 
 (F+tK_X)=0$, one has $t(D\cdot K_X)=0$, and therefore $t=0$. This means that $\phi$ is defined by a linear system of the form $|mF|$, with $m$ a suitable positive integer. But this is impossible because, by hypothesis, $\phi$ contracts $E$ to a point, whereas the map defined by a linear system of the form $|mF|$ maps $E$ to a curve. So we have $D\cdot_ED >0$ and $D=\Gamma$ is smooth, irreducible and rational. 

At this point the proof goes, almost verbatim,  as the one of \cite [Prop. (1.9)]{CF}. We repeat it here for the reader's convenience. 

First we notice that $\Gamma$ is ample on $E$. Indeed, if $A$ is any effective divisor on $E$, one has that both $\Gamma$ and $A$ are numerically equivalent to $C$, so they are numerically equivalent, hence $\gamma \Gamma\equiv A$ on $X$, with $\gamma\in \mathbb Q$ positive. So one has
$$
\Gamma \cdot _E A=F\cdot A=\gamma F\cdot \Gamma=\gamma \Gamma \cdot _E\Gamma>0
$$
thus $\Gamma$ is ample by Nakai's criterion. 

Now consider the surface $F'=\phi(F)$. Since $F'$ has been obtained from $F$ by contracting to $p$ the $(-1)$--curve $\Gamma$, then $F'$ is smooth at $p$. Moreover $Y$ is $\mathbb Q$--factorial. By \cite [Lemma (2.4)]{CF}, $Y$ is smooth at $p$. Let $\mathfrak m$ be the maximal ideal of $Y$ at $p$. Note that
$$
\mathcal O_F(-E)=\mathfrak m \mathcal O_F=\mathcal O_F(-\Gamma)
$$
because $F$ and $E$ intersect transversely along $\Gamma$. Let $\Delta$ be the set of points where $\mathcal O_X(-E)$ is not generated by $\mathfrak m \mathcal O_X$. Since $\Delta$ has empty intersection with the ample divisor $\Gamma$ on $E$, it follows that $\Delta$ is a finite set not lying on $F$. 

Let $X'\longrightarrow Y$ be the blow--up of $Y$ at $p$. By the universal property of the blow--up, there is an induced regular map $X\setminus \Delta \longrightarrow X'$. Let $X^*\longrightarrow X$ be a sequence of blow--ups in $\Delta$ such that the composed map $\eta: X^*\longrightarrow X'$ is a morphism. If $E^*\subset X^*$ is an irreducible component of the exceptional set not equal to the proper transform of $E$, then $\eta(E^*)$ does not intersect the strict transform $F'$ of $F$ in $X'$, since otherwise $E^*$ would intersect the proper transform $F^*\simeq F$ of $F$ on $X^*$, which is not possible. Since  $F'$ is ample when restricted to  the exceptional divisor $E'\cong \mathbb P^2$ of $X'\longrightarrow Y$, we have that $\eta(E^*)$ is a point, and so we get a morphism $X\longrightarrow X'$, that is an isomorphism off a finite set of points. So it is an isomorphism everywhere. But then this contradicts the minimality assumption. 
\end{proof}

\subsection{Sketch of the proof of the contraction theorem}

We can now give the:

\begin{proof} [Sketch of the of the proof of Theorem \ref {thm:mori}] By Lemma  \ref {lem:uno}, there is an extremal contraction $\phi: X\longrightarrow Y$ of a ray $\mathbb R^+[C]$. If $\dim(Y)<3$, we set $(X_0,\L_0):=(X,\L)$. If $\dim(Y)=0$ we are in case (4). 
If $\dim(Y)=1$, then $Y$ is smooth (see \cite [(3.2), (1)]{CF}) and  (1) and (2) follow by \cite [Prop. (3.4)]{CF} that applies verbatim. If $\dim (Y)=2$ then (3) follows from \cite [Prop. (3.3)]{CF} that applies verbatim. Finally, assume that $\phi$ is birational with exceptional set $E$.  By Lemma \ref {lem:4}, $E$ has empty intersection with $F_0$ general in $\L_0$. If $\phi$ is a divisorial contraction, we can replace $X$ with $Y$ without affecting $F_0$ and then we can proceed by induction on the Picard number of $X_0$. If $E$ is a curve, there is a flip $\phi^+: X^+_0\longrightarrow Y$ that does not affect $F_0$. Since sequences of flips terminate, after a finite number of steps we obtain either a divisorial contraction or a contraction of fibre type, finishing the proof. 
\end{proof}

\section{First steps in the proofs} \label{sec:applrat}

In view of the proof of Theorems \ref {thm:main1} or \ref {thm:main3} we make now the following considerations. 

Let $V\subseteq \p^r$, with $r\geq 4$, be an irreducible, projective, non--degenerate, linearly normal variety of dimension 3, such that its general hyperplane section has negative Kodaira dimension. Let $\varphi: X\longrightarrow V$ be a partial desingularization, i.e., it eliminates all singularities that are not terminal and $\mathbb Q$--factorial. Moreover we assume 
that is  \emph{minimal} in this sense: there is no such partial desingularization $g: X'\longrightarrow V$ of $V$ with a birational morphism $h: X\to X'$ such that $\varphi=g \circ h$. Set $\mathcal L=\varphi^*(|\mathcal O_V(1)|)$. Then  $(X,\mathcal L)$ is an R--pair.  

Suppose that $(X,\mathcal L)$  is not minimal. Then by the minimality of the desingularization 
$f: X\longrightarrow V$, the only thing that can happen is that there is pair $(X',\L')$ with a morphism $f: X\longrightarrow X'$ that is the blow--up of a smooth point $p\in X'$ such that the general surface $F'\in \L'$ containing $p$ is smooth with negative Kodaira dimension and $\L$ is the strict transform of the linear system $\L'(-p)$ of surfaces in $\L'$ containing $p$. This clearly implies that there is a variety $V'\subseteq \p^{r+1}$ such that $V$ is the internal projection of $V'$ from a smooth point $q\in V'$. And the general hyperplane section of $V'$ has again negative Kodaira dimension.  If this is the case we may substitute $V$ with $V'$. This \emph{unprojection} process, i.e., passing from $V$ to $V'$, certainly will stop after finitely many steps, because, if we look at it on $X$, this implies the contraction of some exceptional divisor, and therefore reduces the rank of the Picard group of $X$. 

Eventually we will find in the way described above a pair $(X_0,\mathcal L_0)$ that is minimal. At this point we may apply the Contraction Theorem \ref {thm:mori} and we are in one of the cases (1)--(4) of that theorem. The case (4) will be called the \emph{Fano case}.

\begin{proposition}\label{prop:casi} Suppose that in the above setting we are not in the Fano case. If the general hyperplane section of $V$ is rational, then either $V$ is a rational normal scroll as in case (i) of Theorem \ref {thm:main1} or 
we are in one of the cases (ii)--(iv) of Theorem \ref {thm:main1}. If  the general hyperplane section of $V$ has negative Kodaira dimension but is not rational, then we are in one of the cases (a)--(d) of Theorem \ref {thm:main3}.  
\end{proposition}

\begin{proof} Suppose we are in case (1) of Theorem \ref {thm:mori}. Then clearly $V$ is either a scroll in planes or it is swept out by a 1--dimensional family of Veronese surfaces of degree 4 (or external projections of such surfaces). If we are in case (2) of Theorem \ref {thm:mori} then $V$ is swept out by a 1--dimensional family of (generically smooth) 2--dimensional quadrics. If we are in case (3) of Theorem \ref {thm:mori} it is a scroll in lines over a surface that is birational to a hyperplane section  of $V$. The assertion follows.
\end{proof}

 Moreover, we have:

\begin{proposition}\label{prop:rat} Suppose  that in the above setting  we are not in the Fano case. Then $X$ is birational to $F\times \p^1$, with $F\in \L$ general. In particular, if $F$ is rational, then $X$ is rational. 
\end{proposition}

\begin{proof} Suppose we are in case (1) of Theorem \ref {thm:mori}. Then $F\in \L$ is birational to $Y\times \p^1$ and $X$ to $Y\times \p^2$, so $X$ is birational to $F\times \p^1$ as wanted. 

Suppose we are in case (2). Then $F$ is a smooth surface and the general fibre $Z\cong \p^1\times \p^1$ of $\phi: X\longrightarrow Y$ intersects $F$ along a rational curve in $|\O_{\p^1\times \p^1}(1,1)|$, so that $F$ has a pencil $\mathcal P$ of rational curves parametrized by the curve $Y$. It is well known that we can find a unisecant curve $\Gamma$ to the curves of $\mathcal P$, and $\Gamma$ is birational to $Y$. 
Then, in view of the existence of this unisecant,  $F$ is birational to $Y\times \p^1$ and $X$ is birational to $Y\times \p^2$, hence the assertion follows. 

Finally, in case (3), $F\in \L$ is birational to the surface $Y$ via $\phi$ and $X$ is birational to $Y\times \p^1$ and the assertion follows again. \end{proof}

\section{The Fano case}\label{sec:fano}

In this section we deal with the case (4) of Theorem \ref {thm:mori}. So we assume from now on that $(X,\L)$ is a R--pair, with $X$  a $\mathbb Q$--Fano variety with Picard number $\rho(X)=1$.

\subsection{General facts} First we need to recall some more results from \cite {CF}. The following is \cite [Lem. (4.2)]{CF}:

\begin{lemma}\label{lem:5} Let $(X,\L)$ be a R--pair, with $X$  a $\mathbb Q$--Fano variety with Picard number $\rho(X)=1$. If $F\in \L$ is general, then $F$ is a Del Pezzo surface and there exist  positive integers $p,q,s$  such that:\\
\begin{inparaenum}
\item [(i)] $p,q$ are coprime and $pF+qK_X=0$ in ${\rm Pic}(X)\otimes \mathbb Q$;\\
\item [(ii)] $p=q+s$;\\
\item [(iii)] $s=1$ unless $F\cong\p^1\times \p^1$ or $F\cong\p^2$; in the former case  one has $1\leq s\leq 2$, in the latter either $s=1$ or $s=3$. 
\end{inparaenum}
\end{lemma} 

\begin{remark}\label{rem:qfano}
As an immediate consequence of Lemma \ref {lem:5} we have that R--pairs $(X,\L)$ with $X$ a $\mathbb Q$--Fano variety are rational. This concludes the proof of Theorem \ref {thm:main3}, showing that under the hypotheses of that theorem, only cases (1)--(3) of Theorem \ref {thm:mori} may occur.
\end{remark}

The following is \cite [Lem. (4.3)]{CF}:

\begin{lemma}\label{lem:5bis} Let $(X,\L)$ be a R--pair, with $X$  a $\mathbb Q$--Fano variety with Picard number $\rho(X)=1$. Then ${\rm Cl}(X)$ has no torsion. Hence referring to Lemma \ref {lem:5}, (i), one has $pF+qK_X=0$ in ${\rm Cl}(X)$. 
\end{lemma}

The following is \cite [Lem. (4.5)]{CF}:

\begin {lemma}\label{lem:6} Let $(X,\L)$ be a R--pair, with $X$  a $\mathbb Q$--Fano variety with Picard number $\rho(X)=1$. If $D$ is  a Weil divisor on $X$, one has $h^1(X, \O_X(D))=0$. 
\end{lemma}

\subsection{The case $s=1$}
In this case we have the equality of $\mathbb Q$--divisors
$$
K_X= -\frac {q+1}q F
$$
hence
\begin{equation}\label{eq:qu}
K_F=K_X+F_{|F}= -\frac {F_{|F}}q
\end{equation}
and therefore 
\begin{equation}\label{eq:q}
d:=K_F^2= \frac {F^3}{q^2}
\end{equation}
and a priori one has $0<d\leq 9$, since $F$ is a Del Pezzo surface. 

We will consider linear systems of Weil divisors of the form $|F+\mu(F+K_X)|$ and we will denote by $\varphi_\mu$ the rational map determined by such a system. Note that
\begin{equation}\label{eq:qi}
F+\mu(F+K_X)=\Big( 1-\frac \mu q\Big) F
\end{equation}
in ${\rm Cl}(X)\otimes _{\mathbb Z}\mathbb Q$.

From the exact sequence
$$
0\longrightarrow \O_X(\mu(F+K_X))\longrightarrow \O_X(F+\mu(F+K_X))\longrightarrow \O_F(F+\mu K_F)\longrightarrow 0
$$
and from Lemma \ref {lem:6}, we have that $|F+\mu(F+K_X)|$ cuts out on $F$ the complete system 
$F^{(\mu)}:=|F_{|F}+\mu K_F|$, i.e., the \emph{$\mu$--adjoint system} to the \emph{characteristic linear system} $|F_{|F}|$ cut out by $\L$ on $F$, that is itself complete.

Since $F$ is rational, the adjunction extinguishes on $F$, and indeed, the linear systems $|F+\mu(F+K_X)|$ are empty for $\mu>q$ (see \eqref {eq:qi}). For $\mu=q$ the system $|F+\mu(F+K_X)|$ is zero. This implies that $|F+(q-1)(F+K_X)|$ 
cuts out on $F$ the linear system $F^{(q-1)}$ of curves of genus 1, that coincides with the anticanonical system on $F$, so it has dimension $d=K_F^2= \frac {F^3}{q^2}$. Note  in fact that
$$
F+(q-1)(F+K_X)=\frac Fq 
$$
in ${\rm Div}(X)\otimes_{\mathbb Z}\mathbb Q$ and recall \eqref {eq:qu}.

\begin{lemma}\label{lem:bpf} The linear system $|F+(q-1)(F+K_X)|$ has no fixed component and has a base curve $\Gamma$ only if $d=1$ and in this case $\Gamma\cdot F=1$.
\end{lemma}

\begin{proof} This follows from the fact that $|F+(q-1)(F+K_X)|$ cuts out on $F$ (that is ample, because $\rho(X)=1$), the anticanonical linear system that has no fixed components and has  one base point only when $d=1$, because $F$ is Del Pezzo. \end{proof}

In the next subsections we will discuss separately the various cases for $s,q$ and $d$. 

\subsection {The case $s=q=1$}\label{ssec:ell}
In this case the {characteristic linear system} cut out by $\L$ on $F$ is the anticanonical system of $F$ and consists of elliptic curves. Since $\L$ determines a birational map $\phi: X\longrightarrow V\subset \p^r$, one has $d=F^3=\deg(V)\geq 3$ and $d=F^3=\deg(V)\leq 9$. Moreover $r=d+1\geq 4$. 

\begin{lemma}\label{lem:op} In the above set up, if $V$ is singular, then $V$ is rational, hence $X$ is rational.
\end{lemma}

\begin{proof} Suppose $p$ is a singular point of $V$ with multiplicity $m$. Consider the projection of $V$ to $\p^{r-1}$ from $p$, with image $V'$. 

If $V'$ is a surface, then $V$ is a cone with vertex $p$ and $V'$ is isomorphic to a general hyperplane section of $V$, so it is birational to a general $F\in \L$, hence it is rational. Then $V$ itself is rational, and we are done. 

Suppose next that $V'$ is a threefold and that the projection $V\dasharrow V'$ from $p$ has degree $t$. Then $\deg (V')=\frac {d-m}t\geq r-3$, that implies $m=2$, $t=1$. So $V'$ is a variety of minimal  degree $r-3$ in $\p^{r-1}$, hence it is rational, thus $V$, that is birational to $V'$, is rational, and we are done again.\end{proof}

\begin{lemma}\label{lem:op1} In the above set up, if $V$ is smooth, then $V$ is rational, hence $X$ is rational, unless $V$ is a cubic threefold in $\p^4$.
\end{lemma}

\begin{proof} If $V$ is smooth, then $V$ is a Fano threefold of index 2, which is well known to be rational unless $V$ is a smooth cubic threefold in $\p^4$.
\end{proof}

\begin{remark}\label{rem:fano2} Smooth Fano threefolds of index 2 are classified  (see \cite[Chapt. 3]{IP}) and are listed in Theorem \ref {thm:main1} as cases (v1)--(v7). Actually case (v7) does not belong to the case $s=q=1$ but to the case $s=q=2$ as we will later see (see Remark \ref {rem:lhj} below). 

If $V$ as above in this subsection is singular, the proof of Lemma  \ref {lem:op} shows that either $V$ is a cone, in which case it is of type (iv) in the list of Theorem \ref {thm:main1}, or it has at most double points, from each of which it projects birationally to a rational threefold of minimal degree. At the best of our knowledge, a full classification of these singular threefolds is still missing. However in \cite [\S 20]{Enr} Enriques claimed to have the full list of all linear systems of surfaces in $\p^3$ (up to Cremona transformations) which represent the rational threefolds we considered in this section. 

Note that for $d=9$, the variety $V$ is  a cone. In fact it cannot be smooth by the classification of smooth Fano threefolds of index 2  (see \cite[Chapt. 3]{IP}). If it is not a cone, it has some isolated double point $p$. The general hyperplane section of $V$ through $p$ would then be a rational surface (because it projects birationally from $p$ to a surface of minimal degree 7 in $\p^8$) with a double point and general hyperplane section  elliptic curves of degree 9, and this is not possible, because such surfaces are 3--Veronese images of $\p^2$, which are smooth.

\end{remark}

\subsection{The case $s=1, q\geq 2, d\geq 3$}

\begin{lemma}\label{lem:7}  In this set up, consider the rational map $\varphi_{q-1}: X\dasharrow Y$ determined by the linear system $|F+(q-1)(F+K_X)|$. Then $\dim(Y)\geq 2$  and if the equality holds, then $X$ is rational.  
\end{lemma}

\begin{proof}  Since  $d\geq 3$, the linear system $F^{(q-1)}$ of dimension $d\geq 3$ of curves of genus 1 on $F$ is base point free and birational. This implies that $Y$ is at least a surface.
Suppose $Y$ is a surface, and let $Z$ be the general fibre of $\varphi_{q-1}$. Then, since  $F^{(q-1)}$ is birational on $F$, then $Z\cdot F=1$. So $Z$ is a rational curve, and  $Y$ is rational, since it is birational to $F$. Then   $X$ is birational to $Y\times \p^1$ so it is rational. 
\end{proof}

\begin{lemma}\label{lem:8} Same set up of Lemma \ref {lem:7}. Assume $Y$ is a threefold. Let $S\in |F+(q-1)(F+K_X)|$ be a general element. Then the general curve in the  {characteristic  linear system} cut out by $|F+(q-1)(F+K_X)|$ on $S$ is smooth and rational. 
\end{lemma}

\begin{proof} By Lemma \ref {lem:bpf} the {characteristic  linear system} in question has no fixed components. Since we are assuming that $Y$ is a threefold, then  the characteristic  system is not composed with a pencil, hence it consists of irreducible curves.

The adjoint system to the characteristic system is cut out on $S$ by the linear system $|2S+K_X|$. But
$$
2S+K_X=\frac {2F}q-\frac {q+1}q F=-\frac {q-1}qF
$$
that is negative, and this implies  that the {characteristic  system} consists of curves of arithmetic genus 0, and we are done. 
\end{proof}

\begin{proposition}\label{prop:rat} In the set up of Lemma \ref {lem:8}, $X$ is rational. 
\end{proposition} 

\begin{proof} Consider again  the map $\varphi_{q-1}: X\dasharrow Y\subseteq \p^d$. By Lemma \ref {lem:8}, this map is birational, and $Y$ has rational curve sections, so it is a threefold of minimal degree $d-2$ in $\p^d$, so it is rational. The assertion follows.
\end{proof}

\begin{remark}\label{rem:spec} We can be more specific about the values of $d$ and $q$ compatible with the situation considered in this subsection. Indeed,  the proof of Proposition \ref {prop:rat} and \eqref {eq:q} imply that 
$$
\Big (\frac Fq\Big)^3=d-2=\frac {F^3}{q^2}-2.
$$
Hence $F^3(q-1)=2q^3$. This implies that $q^3$ divides $F^3$ so $F^3=aq^3$ for some positive integer $a$. So we get $2=a(q-1)$, and we have only the two possibilities
$$
a=1, q=3, d=3, F^3=27 \quad \text {and} \quad  a=2, q=2, d=4, F^3=16.
$$
These two possibilities correspond to actually existing threefolds. In fact, in the case 
$a=1, q=3, d=3, F^3=27$, the map $\varphi_2: X\dasharrow Y= \p^3$ is birational, and, since
$$
F+2(F+K_X)=\frac F3,
$$
the surfaces $F$ are mapped to cubic surfaces. Hence the image of $\varphi_{\L}$ is the Veronese image of $\p^3$ via the cubics.

In the case $a=2, q=2, d=4, F^3=16$ the map $\varphi_1: X\dasharrow Y\subset  \p^4$ is birational, and $Y$ is a quadric. The surfaces $F$ are mapped to intersections of $Y$ with quadrics. Hence the image of $\varphi_{\L}$ is the Veronese image of a quadric in $\p^4$ via the quadrics. These correspond to types (vi) and (vii) in the list of  Theorem \ref {thm:main1}.
\end{remark}

\subsection{The case $s=1, q\geq 2, d=2$}

Since  $d=2$, the linear system $F^{(q-1)}$, cut out on $F$ by $|F+(q-1)(F+K_X)|$,  of dimension $2$ of curves of genus 1 is base point free and it determines a morphism $\psi: F\longrightarrow \p^2$ of degree $2$, hence  an involution $\iota$ on $F$. This implies that there is a congruence $\mathcal C$ of curves $C$ on $X$ such that:\\
\begin{inparaenum}
\item [$\bullet$] $F\cdot C=2$;\\
\item [$\bullet$] the linear system $|F+(q-1)(F+K_X)|$ determines a dominant map $\varphi_{q-1}: X\dasharrow \p^2$, whose general fibre is a curve in $\mathcal C$.
\end{inparaenum}

Note that
$$
2=F\cdot C=F\cdot \Big (\frac Fq\Big)^2=\frac {F^3}{q^2}
$$
hence
\begin{equation}\label{eq:lpot}
F^3=2q^2.
\end{equation}

\begin{proposition}\label{prop:rarat} In the case $s=1, q\geq 2, d=2$, $X$ is rational. 

\end{proposition}

\begin{proof}
If the curves in $\mathcal C$ are reducible, then they split in two components both intersecting $F$ (that is ample) in one point, and it is clear that $X$ is then rational. So let us assume that the general curve $C$ in $\mathcal C$ is irreducible, hence $X$ in this case has a birational structure of conic bundle over a rational surface. However we have
$$
\frac Fq \cdot C=\frac 2q>0
$$
hence, since $\dim (|F+(q-1)(F+K_X)|)=2$, the linear system $|F+(q-1)(F+K_X)|$ has some base point, which implies that all  curves of $\mathcal C$ pass through some fixed point of $X$. Then the aforementioned conic bundle has a section, hence it is rational. 
\end{proof}

\begin{remark}\label{rem:more} We can be more specific about the threefolds $X$ in this case, according to the value of $q$.\medskip

(a) {\bf Case $q\geq 4$.} The linear system $F^{(q-1)}$  of dimension $2$ of curves of genus 1 on $F$ is birationally equivalent to the linear system $(3;1^7)$ of plane cubics with 7 simple base points such that there is a smooth cubic curve through them. Then the linear system  
$$|2F+2(q-1)(F+K_X)|=|F+(q-2)(F+K_X)|$$ 
cuts out on $F$ a linear system that is birationally equivalent to the linear system $(6;2^7)$ of plane sextics with 7 double base points, that has dimension 6, self--intersection 8, and curves of genus 3. The image of $F$ via the morphism determined by this linear system is a surface $S$ of degree 8 in $\p^6$, that is the complete intersection of a cone $V$ over a Veronese surface of degree 4 in $\p^5$, with a quadric that does not pass through the vertex of $V$.
Note that the lines generating the cone $V$ cut out on $S$ the pairs of points of an involution that coincides with the involution $\iota$ mentioned at the beginning of this subsection. Then we have $\dim (|F+(q-2)(F+K_X)|)=6$. Moreover, we have
\begin{equation*}\label{eq:llppot}
F+(q-2)(F+K_X)=\frac {2F}q.
\end{equation*}
So
$$
(F+(q-2)(F+K_X))\cdot C=\frac {2F}q\cdot C=\frac 4q\leq 1.
$$
Since $|F+(q-2)(F+K_X)|$ cuts out on $S$ a birational linear system, we must have $(F+(q-2)(F+K_X))\cdot C\geq 1$, and this implies  $q=4$ and that the curves of the congruence $\mathcal C$ are mapped by $\varphi_2$ to lines that intersect $S$ in pairs of points conjugated in the involution $\iota$, so they are exactly the lines generating the cone $V$ as we saw above. This proves that the image of $X$ via $\varphi_2$ coincides with the cone $V$ over the Veronese surface in $\p^5$. Then the image on $V$ of the linear system $\L$ on $X$ is just the linear system cut out by the quadrics on $V$. This corresponds to case (viii) in the list of Theorem \ref {thm:main1}. \medskip

(b) {\bf Case $q=3$.} We will prove that this case is not possible. 

One has
$$
2F+4(F+K_X)=2F+K_X=\frac {2F}3
$$
hence $|2F+K_X|$ cuts out on $F$ a linear system that is birationally equivalent to the linear system $(6;2^7)$ as in case (a) and so $\dim  (|2F+K_X|)=6$. Consider the rational map $\varphi_1: X\dasharrow \p^6$ determined by $|2F+K_X|$ and let $V'$ be its image in $\p^6$. Since $F$ is birationally mapped by $\varphi_1$  
to a   surface $S$ of degree $8$ as in case (a), $V'$ is at least a surface. We claim it must be a threefold. Suppose, by contradiction, $V'$ is  a surface and let $\Gamma$ be a general fibre of $\varphi_1: X\dasharrow V$. Since $F$ is mapped birationally to $V'$ by $\varphi_1$, we must have $F\cdot \Gamma=1$, so that $\Gamma$ should be rational. On the other hand, since
$$
2(2F+K_X)+K_X=0,
$$
the characteristic system of $|2F+K_X|$ should be composed with a pencil 
 of curves with arithmetic genus 1, a contradiction.
 
Since
$$
\frac {2F}3\cdot C=\frac 43,
$$
there is some base point of $|2F+K_X|$ sitting on all curves of the congruence $\mathcal C$, that are therefore mapped by $\varphi_1$ to lines, and precisely to the lines joining the pairs of points of $S$ conjugated by the involution $\iota$. As in case (a), this would imply that the image $V'$ of $\varphi_1$ is the cone $V$ over the Veronese surface in $\p^5$, on which $S$ is cut out by a quadric (see case (a)). This would yield the numerical equivalence 
$$
F\equiv 2\cdot (2F+K_X)\equiv  \frac {4F}3
$$
on $X$, a contradiction. \medskip

(c) {\bf Case $q=2$.} In this case we have $\dim(|F|)=7$ and $F^3=8$, and therefore the image $Y$ of $X$ via the morphism $\varphi_\L$  sits in $\p^7$ and its general general hyperplane section is the surface $S$ of degree 8 in $\p^6$ we encountered already in (a) above, that is the complete intersection of a cone $V$ over a Veronese surface of degree 4 in $\p^5$, with a quadric that does not pass through the vertex of $V$. The threefold $Y$ has a congruence $\mathcal C$ of conics. Let us consider the 4--dimensional variety $W\subset \p^7$ that is swept out by the planes of the conics in $\mathcal C$. The general hyperplane section
 of $W$ is exactly the cone $V$, hence $W$ itself is the cone with vertex a line $\ell$ over a Veronese surface of degree 4 in $\p^5$. The variety $Y$ is the intersection of $W$ with a quadric $Q$ that does not contain the line $\ell$. It has two quadruple points at the intersection of $Q$ with $\ell$. This corresponds to case (ix) in the list of Theorem \ref {thm:main1}. 
 
By projecting $Y$ from one of the quadruple points, $Y$ maps birationally to the cone $V$ over the Veronese surface. 
 \end{remark}
 
 \subsection{The case $s=d=1, q\geq 2$}

 In this case the linear system $F^{(q-1)}$, cut out on $F$ by $|F+(q-1)(F+K_X)|$,  is a pencil of curves of genus 1, that is birationally equivalent to the linear system $(3;1^8)$ of plane cubics with 8 simple base points, that has a further simple base point. Hence the one--dimensional base locus of $|F+(q-1)(F+K_X)|$ is a curve $\Gamma$ that is cut out by $F$ in one point. Thus one has
$$
1=F\cdot \Gamma=F\cdot \Big (\frac Fq\Big)^2=\frac {F^3}{q^2}
$$
hence
\begin{equation*}\label{eq:lp}
F^3=q^2.
\end{equation*}

The linear system $|F+(q-2)(F+K_X)|$ cuts out on $F$ a linear system that is birationally equivalent to the linear system $(6;2^8)$ of plane sextics with 8 double base points. This system is base point free, has dimension 3 and it determines a $2:1$ morphism $\eta: F\longrightarrow \Sigma\subset \p^3$ where $\Sigma$ is a quadric cone whose line generators are images of the cubics in $(3;1^8)$. Hence  we have an involution $\iota$ on $F$. Note that  $q=2$ is not possible, since in that case we would have that $|F+(q-2)(F+K_X)|=|F|=\L$ would determine a non--birational map of $X$ to its image, and this is not possible by hypothesis. Moreover there is a congruence $\mathcal C$ of curves $C$ on $X$ such that:\\
\begin{inparaenum}
\item [$\bullet$] $F\cdot C=2$;\\
\item [$\bullet$] the linear system $|F+(q-2)(F+K_X)|$, that has dimension 3,  determines a dominant map $\varphi_{q-2}: X\dasharrow \Sigma\subset \p^3$, whose general fibre is a curve in $\mathcal C$.
\end{inparaenum}

 \begin{proposition}\label{prop:rarat2} In the case $s=d=1, q\geq 3$, $X$ is rational. 

\end{proposition}

\begin{proof} 
If the curves in $\mathcal C$ are reducible, then they split in two components both intersecting $F$ in one point, and it is clear that $X$ is then rational. So let us assume that the general curve $C$ in $\mathcal C$ is irreducible, hence $X$ has a birational structure of conic bundle over a rational surface. We want to prove that this conic bundle has a section, hence it is rational. To see this, we proceed as follows.

On the general surface of $|F+(q-2)(F+K_X)|$, the pencil $|F+(q-1)(F+K_X)|$ cuts out a pencil of curves lying in $\mathcal C$, and the characteristic system of $|F+(q-2)(F+K_X)|$ on a general surface in it consists of pairs of curves lying in this pencil, hence in $\mathcal C$.  One has
$$
\Gamma\cdot (F+(q-2)(F+K_X))= \Gamma\cdot \frac {2F}q=\frac 2q<1
$$
which implies that $\Gamma$ and the general surface of $F+(q-2)(F+K_X)$ must intersect, but do not intersect at a smooth point of $X$, so they have to intersect at some singular point of $X$. On the other hand the intersection points of $\Gamma$ and the general surface of $|F+(q-2)(F+K_X)|$ are also base points of the pencil cut out by $|F+(q-1)(F+K_X)|$ on the general surface of $|F+(q-2)(F+K_X)|$. Since, as we saw, this pencil consists of curves in $\mathcal C$, we see that all curves in $\mathcal C$ must pass through some singular point of $X$, which proves that the aforementioned conic bundle has a section, proving our assertion.
\end{proof}

\begin{remark}\label{rem:jop} As usual we can be more specific about the threefolds $X$ in this case. We suppose that the image of $X$ via $\varphi_\L$ is not a scroll in lines over a rational surface, so that the curves in $\mathcal C$ are generically irreducible. \medskip

(a) {\bf Case $q\geq 4$.} We consider the linear system $|F+(q-3)(F+K_X)|$ that cuts out on $F$ a linear system that is birationally equivalent to the linear system $(9;3^8)$ of plane curves of degree 9 with 8 triple points. This linear system has self--intersection 9, dimension 6, is base point free and birational, so $F$ is mapped via the rational map $\varphi_{q-3}: X\dasharrow \p^6$ to a surface $S$ of degree 9. The surface $S$ has a pencil $\mathcal P$ of plane cubics that are the images of the curves in $(3;1^8)$, that has a base point. Hence the pencil $\mathcal P$ has a base point $p$. The projection of $S$ from $p$ is not birational but $2:1$, and its image is a rational normal scroll $\Sigma$ of degree 4 in $\p^5$, with a line directrix. Thus $F$ sits on the degree 4 cone $V$ over $\Sigma$ with vertex $p$. Note that  the involution $\iota$ on $F$ reads on $S$ as the involution determined by the $2:1$ cover $S\longrightarrow \Sigma$. So the pairs of points conjugated in the involution $\iota$ are cut out on $S$ by the lines through $p$, and on each plane cubic of $\mathcal P$ they  are cut out by the lines through $p$. Therefore these lines fill up the cone $V$.

Next we want to show that the the image of the map $\varphi_{q-3}$ has dimension 3. We argue by contradiction and suppose this is not the case. Since $F$ is birationally mapped to a surface $S$, then the image of $X$ via $\varphi_{q-3}$ must coincide with $S$. Let $\Gamma$ be a general fibre of $\varphi_{q-3}: X\dasharrow S$. Then one must have $\Gamma\cdot F=1$. Note  that 
$-(F+K_X)=F+(q-1)(F+K_X)$ (see Lemma \ref {lem:5}, (i) and (ii)), so that $-(F+K_X)$ is effective. Hence  $|F+(q-3)(F+K_X)|$ contains $|F+(q-2)(F+K_X)|$ with residual $|-(F+K_X)|$. This implies that the map $\varphi_{q-2}:X\dasharrow \Sigma\subset \p^3$ factors through the map $\varphi_{q-3}: X\dasharrow S\subset \p^6$ and a projection $\p^6\dasharrow \p^3$ (actually this must be the projection from a plane spanned by a cubic curve of the pencil $\mathcal P$). But this is impossible since the general fibre $C$ of $\varphi_{q-2}$ is irreducible such that $C\cdot F=2$ whereas the general fibre $\Gamma$ of $\varphi_{q-3}$ is such that $\Gamma\cdot F=1$. 

Next we claim that the image of $X$ via $\varphi_{q-3}$ is just the cone $V$ we met before. To see this, we notice that
$$
(F+(q-3)(F+K_X))\cdot C=\frac {3F}q\cdot C=\frac 6q<2.
$$
This implies that the surfaces in $|(F+(q-3)(F+K_X))|$ intersect the curves in $\mathcal C$ in the base points of this congruence, and in one point off these base points. Hence the curves in $\mathcal C$ are mapped to lines by the map $\varphi_{q-3}$. These are exactly the lines that join pairs of points on $S$ conjugated in the involution $\iota$ and, as we saw above, these lines fill up the cone $V$, as wanted.

Finally we prove that $\varphi_{q-3}: X\dasharrow V$ is birational. Indeed, if $m$ is the degree of this map, we have
$$
\frac {27}q=\Big( \frac {3F}q\Big)^3\geq	4m 
$$
which, since $q\geq 4$, forces $m=1$. 

This yields $q=6$. Indeed, since $V$ has rational curve sections, it cannot be $q\leq 5$, because
$$
F+(q-3)(F+K_X)=\frac {3F}q
$$
and 
\begin{equation}\label{eq:ghorg}
\frac {3F}q+\frac {3F}q+K_X=-\frac {q-5}qF
\end{equation}
which is non--negative if $q\leq 5$. 
 
On the other hand we must have
$$
4\leq \Big( \frac {3F}q\Big)^3= \frac {27}q
$$
so that $q\leq 6$ thus $q=6$. Therefore $F^3=36$ and the characteristic system of $|F|$ on $S$ is birationally equivalent to $(18;6^8)$ that has dimension $21$. So $\dim(|F|)=22$. 

Now notice that
$$
\frac F6\cdot \Big (\frac F2\Big)^2=\frac 32<2.
$$
This implies that $\varphi_{q-3}$ maps the surfaces in the pencil $|F+5(F+K_X)|$ to planes, i.e., exactly to the planes that sweep out the cone $V$. Now
$$
F+3(F+5(F+K_X))=3(F+3(F+K_X))
$$
(see again Lemma \ref {lem:5}, (i) and (ii)), and this implies that $S$ plus three planes generating $V$ are cut out on $V$ by a cubic. In conclusion the image $Y$ of the morphism $\varphi_\L: X\longrightarrow \p^{22}$  is a variety of degree $36$, that is also the image of the cone $V$ via the map determined by the linear system of cubic hypersurface section containing three plane generators of $V$. This is the variety in case (x) of the list of Theorem \ref {thm:main1}. \medskip

(b) {\bf Case $q= 3$.} In this case the linear system $|(F+(q-3)(F+K_X))|$ coincides with $|F|$. We have $F^3=9$ and $\dim(|F|)=7$. The linear system determines a morphism $\varphi_\L: X\longrightarrow Y\subset \p^7$ whose image $Y$ has degree 9, and its general hyperplane section is the surface $S$ that we met  in the discussion of case (a). 

We have
$$
(F+2(F+K_X))\cdot F^2=\frac F3\cdot F^2=3
$$
which implies that the surfaces of the pencil $|F+2(F+K_X)|$ are mapped via $\varphi_\L$ to a pencil $\mathcal R$ of cubic surfaces with elliptic hyperplane sections, so to cubic surfaces in $\p^3$. Remember that $|F+2(F+K_X)|$ has a base curve $\Gamma$ that is mapped by $\varphi_\L$ to a line $\ell$. The family of 3--dimensional subspaces spanned by the cubic surfaces in $\mathcal R$, sweep out a 4--dimensional rational normal cone $T$ of degree 4 containing $Y$. Consider a general hyperplane section $S$ of $Y$. It passes through the intersection point of $\ell$ with the hyperplane $\pi$ cutting out $S$, and $S$ sits on $V$, the intersection of $T$ with $\pi$, that, as we know, is the cone over a quartic rational normal scroll surface $\Sigma\subset \p^5$ with a line directrix. Hence $T$ is the cone with vertex $\ell$ over $\Sigma$. It is then clear that $Y$ is cut out on $T$ by a cubic hypersurface containing three 3--dimensional subspaces generating $T$. This is the variety in case (xi) of the list of Theorem \ref {thm:main1}.
\end{remark}

\subsection{The case $s=2$}
In this case $F\cong \p^1\times \p^1$ and  
$$
K_X= -\frac {q+2}q F
$$
in ${\rm Pic}(X)\otimes _\mathbb Z \mathbb Q$, hence
$$
K_F=K_X+F_{|F}= -\frac {2F_{|F}}q
$$
and therefore 
$$
8=K_F^2= \frac {4F^3}{q^2}
$$
so 
$$
F^3=2q^2. 
$$
\begin {proposition}\label{prop:ccclll} If $s=2$ and $q\geq 3$, then there is a congruence $\mathcal C$ of curves $C$ on $X$ such that $F\cdot C=1$, so that $X$ is rational and the image of $X$ via the map $\varphi_\L$ falls in case (iv) of Theorem \ref {thm:main1}. 
\end{proposition}

\begin{proof} We have
$$
-F-K_X=\frac {2F}q
$$
hence $\dim(|-F-K_X|)=8$ and since
$$
\frac {2F}q+\frac {2F}q+K_X=-\frac {q-2}qF
$$
is negative, then the characteristic system of $|-F-K_X|$ consists of rational curves. If the map $\varphi$ determined by $|-F-K_X|$ is birational to its image in $\p^8$, then this image has degree 6. Then we must have
$$
 \frac {16}q=\Big(\frac {2F}q\Big)^3\geq 6
$$
which yields $q\leq 2$, a contradiction. So the map $\varphi$ is not birational to its image in $\p^8$. Since $|-F-K_X|$ cuts out on $F$ the complete anticanonical system, then the image of $X$ via $\varphi$ is the anticanonical image $S$ of $F$. Let $C$ be the general fibre of $\varphi: X\longrightarrow S$. Since the restriction of $\varphi$ to $F$ is an isomorphism of $F$ with $S$, then $F\cdot C=1$, and the assertion follows. 
\end{proof}

\begin{remark}\label{rem:lhj} The cases $s=2$ and  $1\leq q\leq 2$ do occur and give rise to rational threefolds. If $q=1$, then $F^3=2$ and the image of $X$ via the morphism $\phi_\L$ is  a quadric threefold in $\p^4$, which falls in case (i) of Theorem \ref {thm:main1}. If $q=2$, then $-F-K_X=F$, so the characteristic linear series of $|F|$ is the anticanonical system of $F$ (consisting of elliptic curves), and $F^3=8$. So the image $V$ of $X$ via $\varphi_\L$ is a variety of dimension 3 and degree 8 in $\p^9$ with elliptic curve sections. Then the considerations of \S \ref {ssec:ell} apply. In particular, if $V$  is smooth, it is the 2--Veronese image of $\p^3$ (see case (v7) of Theorem \ref {thm:main1}).
\end{remark}
 
\subsection{The case $s=3$}

In this case $F\cong \p^2$ and  
$$
K_X\equiv -\frac {q+3}q F
$$
hence
$$
K_F=K_X+F_{|F}\equiv -\frac {3F_{|F}}q
$$
and therefore 
$$
9=K_F^2= \frac {9F^3}{q^2}
$$
so 
$$
F^3=q^2. 
$$

\begin {proposition}\label{prop:colll} If $s=3$ and $q\geq 4$, then there is a congruence $\mathcal C$ of curves $C$ on $X$ such that $F\cdot C=1$, so that $X$ is rational and the image of $X$ via the map $\varphi_\L$ falls in case (iv) of Theorem \ref {thm:main1}. 
\end{proposition}

\begin{proof} We have
$$
-F-K_X=\frac {3F}q
$$
hence $\dim(|-F-K_X|)=9$ and since
$$
\frac {3F}q+\frac {3F}q+K_X=-\frac {q-3}qF
$$
is negative, then the characteristic system of $|-F-K_X|$ consists of rational curves. If the map $\varphi$ determined by $|-F-K_X|$ is birational to its image in $\p^9$, then this image has degree 7. So we must have
$$
 \frac {27}q=\Big(\frac {3F}q\Big)^3\geq 7
$$
which yields $q\leq 3$, a contradiction. So the map $\varphi$ is not birational to its image in $\p^9$. Since $|-F-K_X|$ cuts out on $F$ the complete anticanonical system, then the image of $X$ via $\varphi$ is the anticanonical image $S$ of $F$. Let $C$ be the general fibre of $\varphi: X\longrightarrow S$. Since the restriction of $\varphi$ to $F$ is an isomorphism of $F$ with $S$, then $F\cdot C=1$, and the assertion follows. 
\end{proof}

\begin{remark}\label{rem:lhj} The cases $1\leq q\leq 2$ do occur and give rise to rational threefolds. If $q=1$, then $F^3=1$ and $X\cong \p^3$. If $q=2$, then $F^3=4$ and the image $V$ of $X$ via the map $\phi_\L$ has as general hyperplane sections Veronese surfaces, so that $V$ is the cone over the Veronese surface in $\p^6$, which falls in case (i) of Theorem \ref {thm:main1}. 

Let us look at the case $q=3$. Here the characteristic linear series of $|F|$ consists of elliptic curves,  $F^3=9$ and the image $V$ of $X$ via the map $\phi_\L$ has as general hyperplane sections the Veronese surfaces image of the plane via the cubics. Then, as we saw in Remark \ref {rem:fano2}, the only possibility is that $V\subset \p^{10}$ is a cone over the Veronese surface image of the plane via the cubics.\end{remark}

\section{Proof of Theorem \ref {thm:main2}}\label{sec:thm2}

In this section we will prove Theorem \ref {thm:main2}. 
Throughout this section $V\subset \p^{r}$, with $r\geq n+1$, will be an irreducible, projective, non--degenerate, linearly normal variety of dimension $n\geq 4$, such that its general surface linear section is rational. If $V$ is a cone, to prove Theorem \ref {thm:main2} one easily proceeds by induction on the dimension $n$ of $V$. So we may assume that $V$ is not a cone. 
In what follows we will denote by $W$ [resp. by $F$] the general threefold [resp. surface] section of $V$. 

\subsection{Varieties of minimal degree}\label{ssec:ratcurv}

First of all we have the trivial case in which $W$ falls in case (i) of the list of Theorem \ref {thm:main1}. Then $V$ is a variety of minimal degree and therefore it is either a  quadric in $\p^{n+1}$, or a rational normal scroll or a cone in $\p^{n+3}$ with vertex a linear space of dimension $n-3$ over the Veronese surface of degree 4 in $\p^{5}$.  In any case it is rational.

\subsection{Genus one curve sections}\label{ssec:ratcurv}

Next we consider the case in which the general curve section of $V$ has genus 1, i.e., $W$ falls in case (v) of  the list of Theorem \ref {thm:main1}. 

\begin{proposition}\label{prop:ell} If $V\subseteq \p^r$ as above has general curve section of genus 1 and it is not a cone, then either it is singular, in which case it can have only double points, or it is smooth, and in this case it is one of the following:\\
\begin{inparaenum}
\item [(1)] a cubic hypersurface in $\p^{n+1}$;\\
\item [(2)] the complete intersection of two quadrics in $\p^{n+2}$;\\
\item [(3)] the Grassmannian $\mathbb G(1,4)\subset \p^9$ of lines in $\p^4$, or a section of it with a hyperplane or with a subspace of codimension 2, in which case $4\leq n\leq 6$;\\
\item [(4)] the Segre embedding of $\p^2\times \p^2$ in $\p^8$.
\end{inparaenum}

In all cases $V$ is rational except, may be, if it is a smooth cubic hypersurface in $\p^{n+1}$
\end{proposition}

\begin{proof} If $V$ is singular but not a cone, then the same argument as in the proof of Lemma \ref {lem:op} shows that $V$ has at most double points and that it is rational. If $V$ is smooth, the classification in the cases (1)--(4) follows from \cite [Thm. 3.3.1]{IP}. In the cases (2)--(4), $V$ is well known to be rational. In fact, if $V$ is a smooth complete intersection of two quadrics in $\p^{n+2}$, by (birationally) projecting $V$ from a general point on it,  one gets a cubic in $\p^{n+1}$ with some double points, that is rational. The Grassmannian $\mathbb G(1,4)\subset \p^9$
is rational and one has the well known representation of it in $\p^6$ via the linear system of quadrics of $\p^6$ that contain the Segre embedding of $\p^1\times \p^2$ in a hyperplane of $\p^6$ \cite [\S 15] {Se}. Then a hyperplane section of $\mathbb G(1,4)$ is birational to a quadric of $\p^6$, hence it is rational and the intersection of $\mathbb G(1,4)$ with a codimension 2 linear space is birational to a complete intersection of two quadrics in $\p^6$, that is again rational. Finally the Segre embedding of $\p^2\times \p^2$ in $\p^8$ is clearly rational.
\end{proof}

\subsection{Threefold sections as in case (ii) or (iii) of  Theorem \ref {thm:main1}}

\begin{proposition}\label{prop:1dim} If the general threefold section $W$ of $V\subseteq \p^r$ as above is swept out by a 1--dimensional rational family of Veronese surfaces of degree 4 (or external projections of such surfaces) or by a 1--dimensional rational family of (generically smooth) 2--dimensional quadrics, then $V$ is rational.
\end{proposition}

\begin{proof} Suppose first that $W$ is swept out by a 1--dimensional rational family of Veronese surfaces of degree 4 (or external projections of such surfaces). Fix a general surface section $F$ of $V$ spanning a linear space $\Pi$ of dimension $r-n+2$. Let us consider the projection $\pi: V\dasharrow \p^{n-3}$ from $\Pi$, that is dominant. The fibre of $\pi$  passing through a general point $x\in V$  is a threefold section $W_x$ of $V\subseteq \p^r$ containing $F$ (that is a general threefold section of $V$). The Veronese surfaces of the 1--dimensional rational family sweeping out $W_x$ cut out on $F$ a pencil $\mathcal P$ of rational curves that  does not depend on $x$, and we may identify $\mathcal P$ with $\p^1$. Then we can consider the map
$$
\pi': V\dasharrow \p^{n-3}\times \p^1
$$ 
that maps a general point $x\in V$ to the pair $(\pi(x),\phi(x))$, where $\phi(x)\in \p^1\cong\mathcal P$ is the unique point of $\mathcal P$ corresponding to the curve cut out on $F$ by the unique Veronese surface in $W_x$ passing through $x$. The map $\pi'$ is clearly dominant and its fibre over a general point $\xi\in \p^{n-3}\times \p^1$ is a Veronese surface that is isomorphic to $\p^2$ over the field $\mathbb C(\xi)$ of rationality of the point $\xi$  (if the fibre is an external projection of a Veronese,  its normalization is isomorphic to $\p^2$). This implies that $V$ is rational.

Suppose next that $W$ is swept out by a 1--dimensional rational family of (generically smooth) 2--dimensional quadrics. The argument is similar to the previous one. Fix a general surface section $F$ of $V$ spanning a linear space $\Pi$ of dimension $r-n+2$. Let us consider the projection $\pi: V\dasharrow \p^{n-3}$ from $\Pi$. The fibre of $\pi$ passing through a general point $x\in V$  is a threefold section $W_x$ of $V\subseteq \p^r$ containing $F$ (that is a general threefold section of $V$). The quadrics of the 1--dimensional rational family sweeping out $W_x$ cut out on $F$ a pencil $\mathcal P$ of conics that does not depend on $x$, and we may identify $\mathcal P$ with $\p^1$. 
Moreover we can fix a unisecant curve $\Gamma$ on $F$ to the conics of $\mathcal P$. Then consider the map $\pi'$
similar to the one we constructed in the previous case. The map $\pi'$ is  dominant and its general fibre is a quadric. So $\pi'$ endows $V$ with a (birational) structure of a quadric fibration over  $\p^{n-3}\times \p^1$ and, by the existence of the unisecant curve $\Gamma$, this quadric fibration has a unisecant. The rationality of $V$ follows. \end{proof}

\begin{remark}\label{rem:mor} Morin claims in \cite [\S 32]{Mo} that under the hypotheses of Proposition \ref {prop:1dim}, $V$ is swept out by a 1--dimensional rational family of cones of dimension $n-1$ over Veronese surfaces or by a 1--dimensional rational family of  quadrics of dimension $n-1$. Unfortunately Morin's argument is not convincing and  we have not been able to fix it. So we leave it as a problem to check Morin's assertion. 
\end{remark}

\subsection{Scrolls in linear spaces of dimension $n-2$} 

Here we suppose that $W$ falls in case (iv) of the list in Theorem \ref {thm:main1}. 

\begin{proposition}\label{prop:case3} If the general threefold section $W$ of $V\subseteq \p^r$ as above is a scroll in lines over a rational surface, then $V$ is a scroll in linear spaces of dimension $n-2$ over a rational surface, and therefore it is rational.
\end{proposition}

\begin{proof} An easy count of parameters shows that $V$ has a family $\mathcal F$ of lines of dimension $2n-4$. If $x\in V$ is a general point, there is a family $\mathcal F_x$ of dimension $n-3$ of lines in $\mathcal F$ passing through $x$, so that they fill up a variety  $\Pi_x$ of dimension $n-2$. The intersection of $\Pi_x$ with a general linear space of dimension $r-n+3$ passing through $x$ is a line, so $\Pi_x$ is a linear space of dimension $n-2$, as wanted. 
\end{proof}

\subsection{Threefold sections as in case (vi) of Theorem \ref {thm:main1}} 

Before dealing with this case, we recall a result by Fujita (see \cite [Thms. (2.9) and (3.1), and Cor. (3.2)]{Fu}):

\begin{theorem}\label{thm:fuji} Let $X\subset \p^r$ be a normal variety of dimension $n$ and let $Y$ be a smooth hyperplane section of $X$ such that $h^1(Y, T_Y\otimes \mathcal O_Y(-i))=0$ for all integers $i\geq 1$. Then $X$ is a cone with vertex a point over $Y$.  
\end{theorem}

\begin{proposition}\label{prop:fuji} Let $V\subseteq \p^r$ be a variety as above of dimension $n\geq 4$ whose general threefold section is the 3--Veronese embedding of $\p^3$ in $\p^{19}$. Then $V$ is the cone with vertex a linear space of dimension $n-4$ over the 3--Veronese embedding of $\p^3$ in $\p^{19}$ (and therefore it is rational).
\end{proposition}

\begin{proof} We assume first that $V$ has dimension 4. The hypothesis is that a general hyperplane section of $V$ is the 3--Veronese embedding of $\p^3$ in $\p^{19}$. We apply Theorem \ref {thm:fuji} to conclude that $V$ is a cone with vertex a point over the 3--Veronese embedding of $\p^3$ in $\p^{19}$. Indeed, since the 3--Veronese embedding of $\p^3$ in $\p^{19}$ is projectively normal, then also $V$ is projectively normal, so it is normal. To apply Theorem \ref {thm:fuji}  we need to show that
$h^1(\p^3, T_{\p^3}\otimes \O_{\p^3}(-3i))=0$ for all integers $i\geq 1$. From the Euler sequence for $T_{\p^3}$, we deduce that
$$
H^1(\p^3, \O_{\p^3}(1-3i))^{\oplus 4}\longrightarrow H^1(\p^3, T_{\p^3}\otimes \O_{p^3}(-3i))\longrightarrow H^2(\p^3, \O_{\p^3}(-3i))
$$
for all $i\geq 1$.  One has $h^1(\p^3, \O_{\p^3}(1-3i))=h^2(\p^3, \O_{\p^3}(-3i))=0$, for all $i\geq 1$ and this implies $h^1(\p^3, T_{\p^3}\otimes \O_{p^3}(-3i))=0$ for all $i\geq 1$, as wanted. 

Next we assume that $n\geq 5$. The degree of $V$ is 27, and the general 4--fold linear section is a cone over the 3--Veronese embedding of $\p^3$ in $\p^{19}$, so it has a singular point of multiplicity $27$. This implies that $V$ has a set of points of multiplicity 27 that is a linear space of dimension $n-4$. The assertion follows. 
\end{proof}

\subsection{Threefold sections as in case (vii) of Theorem \ref{thm:main1}} 

\begin{proposition}\label{prop:fuji0} Let $V\subseteq \p^r$ be a variety as above of dimension $n\geq 4$ whose general threefold section is the 2--Veronese embedding of a quadric in $\p^4$ in $\p^{13}$. Suppose that $V$ is not a cone. Then $n=4$ and $V$ is the 2--Veronese embedding of $\p^4$ in $\p^{14}$. 
\end{proposition}

\begin{proof} The general curve section $C$ of $V$ sits in $\p^{11}$ and it is the bicanonical image of a curve of genus 5.  The corank of the Gaussian map $\gamma_{C,2K_C}$ is 3 (see \cite [Thm. 1.4]{CiDe}) and by the results in \cite {CiDe} this implies that $V$ is a cone if $n\geq 5$. Moreover  by \cite [Thm. 1.9]{CiDe} there is a universal extension of the bicanonical curve $C$ in $\p^{11}$ and it has dimension 4. By \cite [(9.16)]{CiDe}, it is just the 2--Veronese embedding of $\p^4$ in $\p^{14}$. The assertion follows.  
\end{proof}

\subsection{Threefold sections as in case (viii) of Theorem \ref {thm:main1}}

\begin{proposition}\label{prop:fuji1} Let $V\subseteq \p^r$ be a variety as above of dimension $n\geq 4$ whose general threefold section is the 2--Veronese embedding of the cone (with vertex a point) over a Veronese surface of degree 4 in $\p^5$. Then $V$ is rational. \end{proposition}

\begin{proof} The argument is similar to (and actually easier than) the one of the proof of Proposition \ref {prop:1dim}. Fix a general surface section $F$ of $V$ spanning a linear space $\Pi$ of dimension $r-n+2$. Consider the projection $\pi: V\dasharrow \p^{n-3}$ from $\Pi$. The fibre of $\pi$ over a general point $\xi\in \p^{n-3}$  is a general threefold section $W_\xi$ of $V\subseteq \p^r$. The threefold $W_\xi$ is  rational over the field $\mathbb C(\xi)$ of rationality of the point $\xi$.  This implies that $V$ is rational.
\end{proof}

\begin{remark}\label{rem:morino} Morin claims in \cite [\S 37]{Mo} that a variety as in Proposition \ref 
{prop:fuji1} is in fact a cone. Again Morin's argument is not convincing and  we could not fix it. So we leave it as a problem to check Morin's assertion. 
\end{remark} 

\subsection{Threefold sections as in case (ix) of Theorem \ref {thm:main1}}

\begin{proposition}\label{prop:veroint} Let $V\subseteq \p^r$ be a variety as above of dimension $n\geq 4$ whose general threefold section is as in case (ix) of Theorem \ref {thm:main1}. Then $r=n+4$ and $V$ is the complete intersection of a cone with vertex a linear space $\Pi$ of dimension $n-2$ over a Veronese surface of degree 4 in $\p^5$, with a quadric not containing  $\Pi$.  The variety $V$ is rational. 
\end{proposition}

\begin{proof} The  variety $V$ has degree 8, it sits in $\p^{n+4}$ and its general threefold section $W$ has two points of multiplicity $4$. So $V$ has a variety of dimension $n-3$ and degree 2 of points of multiplicity 4. Take one of these points $x$ and project down $V$ from $x$ to  $\p^{n+3}$. Since $V$ is not a cone with vertex $x$, this projection is birational onto its image $V'$ that has dimension $n$ in $\p^{n+3}$ and degree 4 (see Remark \ref {rem:more}(c)). So $V'$ is a variety of minimal degree, hence it is rational, therefore $V$ is rational. 

We claim that $V'$ is a cone with vertex a linear space $\Pi'$ of dimension $n-3$ over the Veronese surface. In fact the general threefold section $W$ of $V$ through $x$ maps, under the aforementioned projection, to the general threefold section of $V'$, and, on the other hand, it maps to a cone with vertex a point over the Veronese surface (see again Remark \ref {rem:more}(c)). This proves that $V$ itself sits on a cone with vertex a linear space of dimension $n-2$ over a Veronese surface, and the assertion follows. \end{proof}

\subsection{Threefold sections as in case (x) of Theorem \ref {thm:main1}}

\begin{proposition}\label{prop:lpic} Let $V\subseteq \p^r$ be a variety as above of dimension $n\geq 4$ whose general threefold section is as in case (x) of Theorem \ref {thm:main1}. Then $V$ is rational.
\end{proposition}

\begin{proof} The proof goes exactly as the one of Proposition \ref {prop:fuji1}, so we leave it to the reader. 
\end{proof}

\subsection{Threefold sections as in case (xi) in Theorem \ref {thm:main1}}

\begin{proposition}\label{prop:lpic} Let $V\subseteq \p^r$ be a variety as above of dimension $n\geq 4$ whose general threefold section is as in case (xi) in Theorem \ref {thm:main1}. Then $r=n+4$ and $V$ is cut out on a cone with vertex a linear space of dimension $n-2$ over a rational normal scroll surface of degree 4 in $\p^5$ with a line directrix, by  a cubic hypersurface containing three $n$--dimensional linear spaces  generators of the cone.
The variety $V$ is rational. 
\end{proposition}

\begin{proof} To prove the first assertion, we proceed by induction on $n$ for $n\geq 4$. First we  work out the case $n=4$. Let $W$ be a general hyperplane section of $V$. Since $V$ is non--degenerate and linearly normal, the natural restriction map
$$
H^0(\p^8, \mathcal I_{V,\p^8}(2))\longrightarrow H^0(\p^7, \mathcal I_{W,\p^7}(2))
$$
is an isomorphism. Let $Z$ be the intersection of all quadrics in $H^0(\p^8, \mathcal I_{V,\p^8}(2))$. By the above considerations, the general hyperplane section $Z'$ of $Z$ equals the intersection of all quadrics in $H^0(\p^7, \mathcal I_{W,\p^7}(2))$, and this is a cone with vertex a line over a rational normal scroll surface of degree 4 in $\p^5$ with a line directrix. This proves that $Z$ is a cone with vertex a plane over a rational normal scroll surface of degree 4 in $\p^5$ with a line directrix. The group ${\rm Cl}(Z)$ is generated by $H$, the hyperplane section of $Z$, and by $\Pi$, the class of a 4--dimensional linear space generator of the cone $Z$. We can write $V=aH-b\Pi$ in ${\rm Cl}(Z)$. By restricting to a hyperplane section, we see that $a=b=3$. Then the assertion follows.

To prove the induction step one proceeds exactly in the same way as above, so we can skip the details and leave them to the reader.

Finally we have to prove that $V$ is rational. Let $Z$ be the degree 4 rational normal cone of dimension $n+1$ containing $V$, and let $\Phi$ be the cubic hypersurface cutting out $V$ on $Z$ off three $n$--dimensional generators $\Pi_1,\Pi_2,\Pi_3$ containing $\Pi$. The span of $\Pi_1,\Pi_2,\Pi_3$ is easily seen to be a linear space $\Pi'$ of dimension $n+3$, i.e., a hyperplane, containing $\Pi$. This shows that $\Pi'$ is tangent to $\Phi$ all along $\Pi$. A direct computation shows that $\Phi$ has a locus $Q$ of double points contained in $\Pi$, that is scheme theoretically a quadric in $\Pi$.  Let us fix $x\in Q$. The $n$--dimensional generators of the scroll $Z$ cut out on $V$ cubic hypersurfaces of dimension $n-1$ with a double point in $x$ (cut out by $\Phi$), that project down from $x$ birationally to linear spaces of dimension $n-1$. So $V$ projects from $x$ birationally to the scroll $Z'$ swept out by these linear spaces of dimension $n-1$, that all pass through the projection of $\Pi$ that is a linear space of dimension $n-3$. So $Z'$ is a cone over a rational normal scroll surface of degree 4 in $\p^5$ with a line directrix. This shows that $V$ is rational.  

There is a slightly different way to look at the above argument. It is in fact immediate to check that the points of $Q$ have multiplicity 5 for $V$ (they have multiplicity 2 for $\Phi$ and 4 for $Z$, so the multiplicity for the intersection of $\Phi$ with $Z$ is 8, but we have to subtract three generators). So projecting $V$ from such a point, the projection is birational and one gets a variety of dimension $n$ of minimal degree $4$ in $\p^{n+3}$, that is exactly the cone $Z'$ considered above. 
\end{proof}

\section{Proof of Theorem \ref {thm:main4}}\label{sec:04}

We can now give the:
\begin{proof}[Proof of Theorem \ref {thm:main4}]  Suppose first that the general threefold section $W$ of $V$ falls in case (a) of Theorem \ref {thm:main3}. So $W$ contains a 3--dimensional family of lines. Then an easy count of parameters shows that $V$ has a family $\mathcal F$ of lines of dimension $2n-3$. If $x\in V$ is a general point, there is a family $\mathcal F_x$ of dimension $n-2$ of lines in $\mathcal F$ passing through $x$, so that they fill up a variety  $\Pi_x$ of dimension $n-1$. The intersection of $\Pi_x$ with a general linear space of dimension $r-n+3$ passing through $x$ is a plane, so $\Pi_x$ is a linear space of dimension $n-1$. Hence $V$ is a scroll in linear spaces of dimension $n-1$ over a curve $\Gamma$ and therefore $V$ is birational to $\Gamma\times \p^{n-1}$. On the other hand, the general surface section $F$ of $V$ is a scroll in lines over $\Gamma$, so it is birational  to $\Gamma\times \p^1$. It follows that $V$ is birational to $F\times \p^{n-2}$, as wanted.

Suppose next that the general threefold section $W$ of $V$ falls in case (b) or (c)  of Theorem \ref {thm:main3}. Arguing exactly as in the  proof of Proposition \ref {prop:1dim}, one can prove that $V$ is birational to $F\times \p^{n-2}$. We can therefore omit the details, leaving them to the reader.

Finally,  suppose that the general threefold section $W$ of $V$ falls in case (d)  of Theorem \ref {thm:main3}. Arguing exactly as in the proof of Proposition \ref {prop:case3}, one checks that $V$ is a scroll in linear spaces of dimension $n-2$ over a surface, and the assertion follows.
\end{proof}

\end{document}